\documentclass{llncs}
\usepackage{verbatim}
\usepackage{graphicx}
\usepackage{amsfonts}
\usepackage{amscd}
\usepackage{amssymb}
\usepackage{amsmath}

\usepackage{alltt}
\usepackage{rotating}
\usepackage{floatpag}
 \rotfloatpagestyle{empty}
\usepackage{graphicx}
\usepackage{multind}
\usepackage{times}

\usepackage{verbatim}
\usepackage{latexsym}
\usepackage{crop}
\usepackage{txfonts}
\usepackage[hyphens]{url}
\usepackage{setspace}
\usepackage{ellipsis} %

\usepackage{amsfonts}
\usepackage{amsmath}
\usepackage{tikz} 
\usetikzlibrary{chains,shapes,arrows,shadings,trees,matrix,positioning,intersections,decorations,
  decorations.markings,backgrounds,fit,calc,fadings,decorations.pathreplacing}

\usepackage[mathscr]{euscript} 
\usepackage{pifont} 
\usepackage[displaymath]{lineno}
\usepackage{rotating}

\def\op#1{{\hbox{#1}}} 

\def\tikzfig#1#2#3{%
\begin{figure}[htb]%
  \centering
\begin{tikzpicture}#3
\end{tikzpicture}
  \caption{#2}
  \label{fig:#1}%
\end{figure}%
}

  {~\par\phantom{!}\endgroup\bigskip}
\def\wasitemize{\relax}

\def\case#1{{\sc (#1)}}
\def\firstcase#1{{\indy{Index}{named property!#1}}\relax{\sc (#1)}}
\def\claim#1{{\it  #1}}

\def\indy#1#2{\index{index/#1}{#2}\relax}
\def\newterm#1{\indy{Index}{#1}{\it #1}\relax}

\def\fullterm#1#2{\indy{Index}{#2}{\it #1}\relax}
\def\cc#1#2{%
  \indy{Index}{computer calculation!{#1}}{\it computer calculation}
  {\footnote{\guidfoot{#1}  #2}}~\cite{website:FlyspeckProject}} 

\def\hide#1{}
\def\swallowed{\relax}
\def\swallow#1\swallowed{}
\newenvironment{iproved}{}{}

\def\hideproof{\renewenvironment{iproved}{%
   \centerline{\it -- Proof Proofed --}
  
  \renewenvironment{enumerate}{}{}
  \def\item{\relax}
  \catcode13=12
  \swallow
}{}}
\def\showproof{\renewenvironment{iproved}{\begin{proof}}{\end{proof}}}
\def\resetproved{\if\displayallproof t\showproof\else\hideproof\fi}

\def\rating#1{\if\displayrating t%
  {{\textsc {[rating={\ensuremath {#1}}].\ }}}\else{}\fi}

\def\cutrate{}
\def\oldrating#1{\if\displayrating t%
  {{\textsc {[former rating={\ensuremath {#1}}].\ }}}\else{}\fi}

\def\formalauthor#1{\if\verbose t{{\tt [formal proof by #1].\ }}\else{}\fi}
\def\dcg#1#2{{\if\verbose t%
  {{\tt{[DCG-#1]}}\indy{References}{ZC{#2 #1}@{DCG-#1}|page{#2}}}\else{}\fi}}
\def\tlabel#1{\label{#1}\if\verbose t{{\tt [#1].\ }%
   \indy{References}{#1|itt}}\else{}\fi}
\def\ifcverbose#1#2{\if\verbose t{{#1}}\else{#2}\fi}
\def\ifverbose#1{\ifcverbose{#1}{}}  
\def\formal#1{\relax }

\def\guid#1{\ifverbose{{\tt [#1]}}}
\def\guidfoot#1{{{\tt [#1]}}}

\def\mar#1{}

\def\textand{\text{ \ and \ }}  

\def\emptyset{\varnothing}

\def\|{\hbox{\ensuremath{\hspace{0.1em}|\hspace{-0.1em}|\hspace{0.1em}}}}
\def\mid{\ :\ }

\def\norm#1#2{\hbox{\ensuremath{\|#1\unskip-\unskip{#2}\|}}}
\def\normo#1{{\|#1\|}}

\def\leftclosed{[}
\def\rightopen{)}

\def\=#1{\accent"16 #1}

\def\CalV{{\mathcal V}}
\def\CalL{{\mathcal L}}

\newcommand{\ring}[1]{\mathbb{#1}}

\def\v{{\mathbf v}}
\def\u{{\mathbf u}}
\def\w{{\mathbf w}}
  
\def\p{{\mathbf p}}

\def\op#1{{\operatorname{#1}}}

\def\azim{\op{azim}}

\def\sol{\operatorname{sol}}

\def\orz{{\mathbf 0}} 
\def\Wdarto{W^0_{\text{dart}}}

\def\card{\op{card}}

\def\pqr#1{#1} 

\def\ee{\varepsilonup}
\def\ocirc{}

\def\verbose{f}

\begin{document}

\title{A Proof of Fejes T\'oth's  Conjecture on Sphere Packings with
Kissing Number Twelve}
\author{Thomas C. Hales\thanks{{Research supported in part by 
NSF grant 0804189 and the Benter Foundation.}}}
\institute{University of Pittsburgh\\
\email{hales@pitt.edu}}
\maketitle

\section{Fejes T\'oth's Conjecture}

On December 26, 1994, L. Fejes T\'oth wrote to me, ``I suppose that
you will be interested in the following conjecture: In $3$-space any
packing of equal balls such that each ball is touched by twelve others
consists of hexagonal layers.  In the enclosed papers a strategy is
described to prove this conjecture''~\cite{Fejes-Toth:89},
~\cite{Fejes-Toth:69}.  This article verifies Fejes-T\'oth's
conjecture.\footnote{The results of this article were presented at the
  Fields Institute, September 2011.}

A packing of balls in this article is identified with its set of
centers.  We adopt the convention that balls in a packing have unit
radius.  Formally, a \newterm{packing} is a set $V\subset\ring{R}^3$
for which $\norm{\v}{\w} < 2$ implies $\v=\w$,
for every $\v,\w\in V$.  It is known that the kissing number
in three dimensions is twelve~\cite{Leech:1956:MG}.  Call a nonempty
packing $V$ in $\ring{R}^3$ in which every $\v\in V$ has distance $2$
from twelve other $\w\in V$ a \newterm{packing with kissing number twelve
}.  

Two examples of packings consisting of hexagonal layers in which each
ball is touched by twelve others are the face-centered cubic packing
(FCC) and the hexagonal-close packing (HCP).  In the FCC packing, the
arrangement of twelve other balls around each ball is identical.  We
call this particular arrangement of twelve balls the FCC pattern.
Similarly, in the HCP packing, the arrangement of twelve around each
ball is identical, and we call this particular arrangement the HCP
pattern.


It is well known that if the arrangement around each ball is either
the FCC or HCP pattern, then the packing consists of hexagonal
layers~\cite[Sec.~1.3]{DSP}.  Hence the following theorem, which is
the main result of this article, is enough to guarantee that a packing
with kissing number twelve consists of hexagonal layers.

\indy{Notation}{V@$V$ (packing)}%

\begin{theorem}[Packings with kissing number twelve]\guid{BDEDUTL}\label{thm:fc} 
  Let $V$ be a packing with kissing number twelve.  Then for every point $\u\in
  V$, the set of twelve around that point is arranged in the pattern
  of the HCP or FCC packing.
\end{theorem}
\indy{Index}{HCP}%
\indy{Index}{FCC}%
\indy{Index}{kissing configuration}%

The truncation parameter $h_0=1.26$ will be used throughout this article.
The following estimate is one of the main results of~\cite{DSP}.  Its proof
is omitted.

\begin{lemma}\label{lemma:dsp}
Let $V$ be a finite packing.  Assume that $2\le \normo{\v}\le 2h_0$ for
every $\v\in V$.  Let $L(h) = (h_0-h)/(h_0-1)$.  Then
\begin{equation}\label{eqn:L}
\sum_{\v\in V} L(\normo{\v}/2) \le 12.
\end{equation}
\end{lemma}

\begin{lemma}\guid{LIHVTRE} \label{lemma:gap}
  Let $V$ be any packing with kissing number twelve and let $\u,\v\in V$.
  Then $\u=\v$, $\norm{\u}{\v}=2$, or $\norm{\u}{\v} \ge 2h_0$.
\end{lemma}

\begin{proof} Let $ \u_1,\ldots, \u_{12}$ be the twelve kissing points
  around $\u$; that is, $\norm{\u_i}{\u}=2$.  
By translating $V$, we may assume without
loss of generality that $\u=\orz$.  Assume that $\v\ne \u_i,\orz$ is in $V$.   
By Lemma~\ref{lemma:dsp},
\[
   L(\normo{\v}/2)  + 12 
  =  L(\normo{\v}/2) + \sum_{i=1}^{12} L(\normo{\u_i}/2)  \le 12.
\]
This implies that $L(\normo{\v}/2)\le 0$, so $\normo{ \v}\ge 2h_0$.
\qed
\end{proof}

A packing $V$ may always be translated so that $\orz\in V$.  We study
the structure of a kissing configuration centered at the origin that
has the separation property of Lemma~\ref{lemma:gap}.

\begin{definition}[$S^2(2)$,~$\CalV$]
  Let $S^2(2)$ be the sphere of radius $2$, centered at the origin.  Let
  $\CalV$ be the set of packings $V\subset \ring{R}^3$ such that
\begin{enumerate}\wasitemize 
\item $\card(V) = 12$,
\item $V\subset S^2(2)$,
\item For all $\u,\v\in V$, we have $\u=\v$, $\norm{\u}{\v}=2$,
or $\norm{\u}{\v}\ge 2h_0$.
\end{enumerate}\wasitemize 
\indy{Notation}{V4@$\CalV$ (twelve sphere configurations)}%
\indy{Notation}{S@$S^2(r)$ (sphere of radius $r$)}%
\end{definition}

If $V\in \CalV$, set 
\[
E_2(V)=\{\{\v,\w\} \mid \norm{\v}{\w}=2\},
\]
the set of {\it contact} edges.

\section{Definitions and Review}

We follow the general approach to sphere packing problems described in
\cite{DSP}.  We use a number of definitions given there.  In
particular, we have the following.

\begin{definition}[affine]\guid{BYFLKYM}\label{def:aff} 
  If $S = \{\v_1,\v_2,\ldots,\v_k\}$ and $S'=\{\v_{k+1},\ldots,\v_n\}$
  are finite subsets of $\ring{R}^N$, then set
	\begin{align*}
\op{aff}_{\pm} (S,S') &= \{t_1 \v_1 +\cdots t_n \v_n \mid
	t_1 +\cdots+t_n = 1,~ \pm t_j \ge 0, \text{ for } j>k\},\\
\op{aff}^0_{\pm} (S,S') &= \{t_1 \v_1 +\cdots t_n \v_n \mid
	t_1 +\cdots+t_n = 1,~ \pm t_j > 0, \text{ for } j>k\}.
		\end{align*}
\indy{Notation}{aff2@$\op{aff}_{\pm}$, $\op{aff}^0_{\pm}$}%
\indy{Index}{affine}%
\indy{Notation}{S@$S\subset\ring{R}^n$}%
\end{definition}


This notation is general enough to describe rays, lines, open and closed intervals,
convex hulls, and affine hulls.
To lighten the notation for singleton sets, abbreviate
$\op{aff}_\pm(\{\v\},S')$ to $\op{aff}_\pm(\v,S')$.
If $S\subset\ring{R}^3$ is a finite set of points,
abbreviate
\begin{align*}
C_\pm(S) &= \op{aff}_\pm(\orz,S),\\
C^0_\pm(S) &= \op{aff}^0_\pm(\orz,S).\\
\end{align*}
When the subscript is absent, the subscript $+$ is implied: $C_+(S)
= C(S)$, and so forth.  The parentheses around the set are frequently
omitted without change in meaning:
\[ C^0\{\v,\w\}=C^0_+(\{\v,\w\}) =
\op{aff}^0_+(\{\orz\},\{\v,\w\}) = \op{aff}_+^0(\orz,\{\v,\w\}).\]

\begin{definition}[fan,~blade]\label{def:fan}
Let $(V,E)$ be a pair consisting of a set $V\subset \ring{R}^3$ and
a set $E$ of unordered pairs of distinct elements of $V$.  The pair
is said to be a \newterm{fan\/} if the following properties hold.
\begin{enumerate}
\item \firstcase{cardinality} $V$ is finite and  nonempty.
\item \firstcase{origin} $\orz\not\in V$.
\item \firstcase{nonparallel} If $\{\v,\w\} \in E$, then the set
$\{\orz,\v,\w\}$ is not collinear.
\item \firstcase{intersection}
For all $\ee ,\ee '\in E \cup \{\{\v\}\mid \v\in V\}$, 
\[ C(\ee )\cap C(\ee ') = C(\ee \cap \ee
').\] 
\end{enumerate}
When $\ee\in E$, call $C^0(\ee)$ or $C(\ee)$ a \newterm{blade}
of the fan.  
\end{definition}

Basic properties of fans are developed in \cite[Ch.~5]{DSP}.
Every fan is graph with vertex set $V$ and edge set $E$.
We use the terminology of graph theory to describe fans.  For example,
we say that $\v,\w$ are \newterm{adjacent} if $\{\v,\w\}\in E$.

\begin{definition}[hypermap]   A \newterm{hypermap} is a finite set $D$, together with three functions
  $e,n,f:D\to D$ that compose to the identity
  \[ 
e\ocirc n\ocirc f = I_D.
\]  The
elements of $D$ are called \fullterm{darts}{dart}.  The functions $e,n$ and
$f$ are called the \fullterm{edge map}{edge!map}, the \fullterm{node map}{node!map}, and
the \fullterm{face map}{face!map}, respectively. 
The orbits of the set of darts under the edge map, node map, and face map are called edges, faces, and nodes of the hypermap.
\end{definition}

\begin{example}[dihedral]\label{ex:D2k} 
There is a hypermap $\op{Dih}_{2k}$ with a dart set of cardinality $2k$.
The permutations $f,n,e$ have  orders $k$, $2$, and $2$ respectively, and 
$e n f = I$.
The set of darts is given by
\[ 
\{x, f x,f^2 x,\ldots,f^{k-1} x\}\cup \{n x, n f x, n f^2 x,\ldots, n f^{k-1} x\}
\] 
for any dart $x$.
If a hypermap is isomorphic to $\op{Dih}_{2k}$ for
some $k$, then it is \newterm{dihedral}.
(The three permutations generate the dihedral
group of order $2k$, acting on a set of $2k$ darts under the left action of the group upon itself.) 
\indy{Index}{hypermap!dihedral}%
\end{example}


Basic properties of hypermaps are developed in \cite[Ch.~4]{DSP}.  An
\newterm{isomorphism of hypermaps} is a bijection $\phi:D\mapsto D'$
such that $\phi(h x) = h'\phi x$, for all $x\in D$ and all structure
permutations $(h,h')=(e,e'),(f,f'),(n,n')$.

For each $\v\in V$ such that $\card(E(\v))>1$, we define a cyclic
permutation $\sigma(\v):E(\v)\to E(\v)$ as follows.  Choose
$\pi_\v:\ring{R}^3\to\ring{C}$, a surjective real-linear  map, with
$\v\in \op{ker}\,\pi_\v$, chosen so that
$(\pi^{-1}_\v(1),\pi^{-1}_\v(\sqrt{-1}),\v/\normo{\v})$ is a
positively-oriented, orthonormal basis of $\ring{R}^3$.  Define the
permutation $\sigma(\v)$ by pulling back the counterclockwise cyclic
permutation of the points on the unit circle:
\[
\{\pi_\v(\w)/\normo{\pi_\v(\w)} \mid \w\in E(\v)\} \subset \ring{C}^\times.
\]
Write $\sigma(\v,\w)$ for $\sigma(\v)(\w)\in E(\v)$.  
Let $\op{arg}(r e^{i\theta}) =\theta\in\leftclosed0,2\pi\rightopen$ 
be the argument of a complex number, and define the azimuth angle by
\[
\op{azim}(\orz,\v,\u,\w) = \op{arg}(\pi_\v(\w)/\pi_\v(\u)).
\]
When $x=(\v,\w)\in D_1$, set 
\[
 \op{azim}(x) = \op{azim}(\orz,\v,\w,\sigma(\v,\w)).
\]
Also, set
\[
 \Wdarto(x) = \{\p\in\ring{R}^3 \mid 0 < 
 \azim(\orz,\v,\w,\p) < \azim(\orz,\v,\w,\sigma(\v,\w)) \}.
\]

We may associate a hypermap with a fan by the following
construction.  Let $(V,E)$ be a fan.  Define a set of darts $D$ to be
the disjoint union of two sets $D_1,D_2$:
\begin{align*}
D_1 &= \{(\v,\w)\mid \{\v,\w\}\in E\},\\
D_2 &= \{(\v,\v) \mid \v\in V,\ \ E(\v) = \emptyset\},\textand \\
D\phantom{_2}   &= D_1\cup D_2,
\end{align*}
where $E(\v)=\{\w \mid \{\v,\w\}\in E\}$, the set of darts adjacent to $\v$.
Darts in $D_2$ are said to be \newterm{isolated} and darts in
$D_1$ are \fullterm{nonisolated}{isolated}.
Define permutations $n$, $e$, and $f$ on $D_1$ by
\begin{align*}n(\v,\w) &= (\v,\sigma(\v,\w)),\\
f (\v,\w) &= (\w,\sigma(\w)^{-1} \v),\\
e (\v,\w) &= (\w,\v).
\end{align*}
Define permutations $n,e,f$ on $D_2$ by making them degenerate on $D_2$:
\[ 
n (\v) = e(\v) = f(\v) = \v.
\] 
Set 
$\op{hyp}(V,E)=(D,e,n,f)$. 
If $(V,E)$ is a fan, 
then it is known that $\op{hyp}(V,E)$ is a hypermap~\cite{DSP}.

If $(V,E)$ is a hypermap, then we
define $X(V,E)$ to be the union of the sets $C(\ee)$, for $\ee\in E$,
and $Y(V,E) = \ring{R}^3 \setminus X(V,E)$ to be its complement.  There is
a well-defined mapping, $F\to U_F$, from the set of faces of $\op{hyp}(V,E)$ to the set of 
topological
connected components of $Y(V,E)$.  The component $U_F$ is characterized by the condition
\[
U_F \cap B(\v,\epsilon) \cap \Wdarto(x)\ne\emptyset,
\]
for all $x=(\v,\w)\in F$ and all sufficiently small $\epsilon>0$ (where $B(\v,\epsilon)$ is
an open ball of radius $\epsilon$ at $\v$).
 We write $\op{sol}(U_F)$ for the solid
angle of $U_F$ at $\orz$.  The sum of the solid angles of the topological components
of $Y(V,E)$ is $4\pi$.

\begin{definition}[local fan]\guid{FTNGOGF} \label{def:convex-local}
A triple $(V,E,F)$ is a \newterm{local fan} if the following conditions hold.
\begin{enumerate} 
\item $(V,E)$ is a fan.
\item  $F$ is a face of $H = \op{hyp}(V,E)$.
\item  $H$ is isomorphic to $\op{Dih}_{2k}$, where $k =
\card(F)$.
\end{enumerate}
\end{definition}
\indy{Index}{fan!local}%

\begin{definition}[localization]\guid{BIFQATK}
\hspace{-3pt}
 Let $(V,E)$ be a fan and let $F$ be
a face of $\op{hyp}(V,E)$.  Let
\begin{align*}
V' &= \{\v\in V \mid \exists~\w\in V.~~(\v,\w)\in F\}.\\
E' &= \{\{\v,\w\} \in E\mid (\v,\w)\in F\}.
\end{align*}
The triple $(V',E',F)$ is called the \newterm{localization} of $(V,E)$ along $F$.
\end{definition}
\indy{Index}{localization}%

The localization is a local fan.

\section{Strategies}

The strategy of the proof is to classify the hypermaps of 
fans $(V,E_2(V))$ for $V\in \CalV$ and to show that there are only two
possibilities: the contact hypermaps of the FCC and the HCP.  From
this, the proof of Fejes T\'oth's conjecture ensues.

The classification result is analogous to the one that we have 
obtained for tame hypermaps in \cite{DSP}.  
This suggests developing a proof along
exactly the same lines.  We define a 
collection of hypermaps with properties that are analogous to those
defining a tame hypermap and call them hypermaps with 
\fullterm{tame
  contact}{tame!contact}.  A computer generated classification of these hypermaps
gives only a few possibilities.  Those other than the FCC and
HCP hypermaps are eliminated by linear programming methods.

\begin{remark}[Lexell's theorem]
According to Lexell's theorem, for any two distinct non-antipodal points $\u,\v\in S^2(2)$,
the locus of points $\w\in S^2(2)$, along which
the spherical triangle with vertices $\u,\v,\w$ has given fixed area, is a circular arc
with endpoints at the antipodes of $\u$ and $\v$.

Lexell's theorem is an aid in finding the minimum of
  $\sol(U_F)$.
  It is a consequence of Lexell's theorem that the area of a spherical
  triangle (viewed as a function of its edge lengths) does not have a
  interior point local minimum, when the edge lengths are
  constrained to lie in given intervals.  For the spherical triangle of minimal area,
each  edge length is extremal.
\indy{Index}{Lexell's theorem}%
\end{remark}

\begin{remark}[Leech's solution of the Newton--Gregory problem]
During a famous discussion with Gregory, Newton asserted that 
if $V\subset S^2(2)$ is any packing, then $\card(V)\le 12$. That is, at most
twelve nonoverlapping balls can touch a fixed central ball.  
\indy{Index}{Newton--Gregory problem}%
\indy{Index}{Leech}%
\indy{Index}{thirteen-spheres problem}%
  
 Leech's proof of Newton's assertion is
noteworthy~\cite{Leech:1956:MG}.  Assuming the existence of a packing
$V\subset S^2(2)$ of cardinality thirteen, Leech associates a planar graph
$(V,E)$ with $V$, which is similar to our standard fan.  In
our notation, he estimates the solid angle of each topological
component $U_F$.  These solid angle estimates can be verified with Lexell's theorem.
Next, he classifies the planar graphs $(V,E)$ that satisfy various
combinatorial constraints obtained from the solid angle estimates. He
finds that no such planar graph exists, in confirmation of Newton's
assertion.
\end{remark}

\section{Main Estimate}

Let $(V,E)$ be a fan and let $F$ be a face of $\op{hyp}(V,E)$.
When $x=(\v,\w)$ is a dart in $\op{hyp}(V,E)$, define $h(x) = \normo{\v}/2$.
Define the weight function
\begin{equation}
  \tau(V,E,F) 
  = \sol(U_F) + (2- k(F))\sol_0
\label{eqn:tau-sol}
\end{equation}
where $\sol_0=3\arccos(1/3)-\pi\approx 0.55$ is the solid angle of a
spherical equilateral triangle with a side of arclength $\pi/3$, and
$k(F)$ is the cardinality of $F$.

\begin{definition}[$\CalV'$,~$E^+(V)$]\guid{ZTONFGU}
Let $V$ be a packing in $\ring{R}^3$.
  Write
$E^+(V)$ for the set of pairs $\{\u,\v\}\subset V$
  such that $2\le\norm{\u}{\v} <\sqrt8$.
Let $\CalV'\subset\CalV$ be the subset of packings $V\in\CalV$ that
satisfy the following two properties. 
\begin{enumerate}
\item  $(V,E^+(V))$ is a fan.
\item The graph $(V,E^+(V))$ is biconnected.
\end{enumerate}
\end{definition}
\indy{Notation}{V4@$\CalV'$ (twelve sphere configurations)}%
\indy{Notation}{E2@$E^+(V)$}%

The next theorem is the main estimate for packings with kissing number twelve.

\begin{theorem}[main estimate]\guid{VGJDQJV}\label{lemma:main-estimate-12}
Let $V'\in \CalV'$.  
Let $F$ be a face of $(V',E^+(V'))$ with at least three darts, 
and let $(V,E,F)$ be the localization of $(V',E^+(V'))$ along $F$.
Let $U=U_F$ be the topological component of
  $Y(V,E)$ corresponding to $F$. 
Let
  $S$ be the subset of $E$ consisting of edges $\{\v,\w\}$ such that $2h_0\le\norm{\v}{\w}<\sqrt8$.
  Set $r=\card(E\setminus S)$ and $s = \card(S)$.
Then
\begin{equation}\label{eqn:tau-tgt}
\tau(V,E,F) \ge \min(d_2(r,s),\op{tgt}),
\end{equation}
where 
\[ d_2(r,s) = 
\begin{cases}
0, & \text{if } (r,s)=(3,0),\\
0.103 (2-s) + 0.27 (r+2s-4), & \text{otherwise.}
\end{cases}
\] 
and
$\op{tgt}=1.541$.
\end{theorem}
\indy{Notation}{s@$s=\card(S)$}%
\indy{Notation}{tgt@$\op{tgt}=1.541$}%

\begin{proof}  
  We take the spherical Delaunay triangulation of the sphere $S^2(2)$ induced by $V'$.  
Triangles correspond to triangulated faces of the polyhedron obtained as 
the convex hull of $V'\subset\ring{R}^3$.  By standard estimates 
\cite{DSP}, by the length constraint $\norm{\v}{\w}<\sqrt8$, 
each edge of $E^+(V')$ is an edge of this polyhedron and 
gives an edge of a Delaunay triangle.  Thus $U_F$, up to a set of measure zero, 
is a disjoint union of cones over Delaunay triangles:
\[
U_F \approx \bigcup \op{aff}_+^0(\orz,\{\u_1,\u_2,\u_3\}).
\]
We may calculate the solid angle of $U_F$ and also $\tau(U_F)$ by this triangulation.

By the kissing number problem, $V'$, which has cardinality $12$, is a saturated spherical network on $S^2(2)$;
that is, there is no room to add a further point on $S^2(2)$ that has distance at least $2$ from all points
of $V'$.  By this saturation property, if $\{\u_1,\u_2,\u_3\}\subset V'$ is a Delaunay triangle, then
the circumradius of the simplex $\{\orz,\u_1,\u_2,\u_3\}$ is less than $2$.  Since $\u_i\in S^2(2)$, this
corresponds to a Euclidean triangular circumradius less than $\sqrt3$ for $\{\u_1,\u_2,\u_3\}$.

By construction,  every edge in $E\setminus S$ has length $2$.
Edges of Delaunay triangles that are not in $E$ have length at least $\sqrt8$.  The upper bound
on these edges will be determined by the circumradius constraint.

Fix attention on a single Delaunay triangle $\{\u_1,\u_2,\u_3\}$. Shifting notation,
let $r$ be the number of edges of the triangle of length $2$.  Let $s$ be the number of edges
of length in the range $\leftclosed 2h_0,3.0\rightopen$.  Let $t$ be the remaining number of edges; 
that is, the number of those of length at least $3.0$.  We have $r+s+t=3$.

Define $d_3:\ring{N}^3\to\ring{R}$ by
\[
d_3(r,s,t) = 
\begin{cases}
0, & \text{if } (r,s,t)=(3,0,0),\\
0, &\text{if } (r,s,t)=(2,0,1)\\
0.103 (2 - s) + 0.27 (r+2 s+2 t-4),&\text{otherwise. }
\end{cases}
\]
Note that $d_3(r,s,0)=d_2(r,s)$.

We claim that the solid angle $A=\op{sol}(\orz,\op{aff}_+^0(\orz,\{\u_1,\u_2,\u_3\}))$ satisfies
\begin{equation}\label{eqn:d3}
A\ge \sol_0 + d_3(r,s,t).
\end{equation}
To show this, we work case by case in the parameters $(r,s,t)$.
Suppose first that $(r,s,t)\ne (1,1,1)$.
If $t>0$, we use the triangle circumradius bound $\sqrt3$ 
to give an upper bound $b_{r,s,t}$ on the edges of length at least $3.0$.
Then we use Lexell's theorem to calculate bounds on the area of the Delaunay triangle,
given the ranges on edges, and in each case we find that the area estimate is satisfied.

We work a few cases explicitly.
For example if $(r,s,t)=(3,0,0)$, then the solid angle is exactly $\sol_0$ and the constant $d_3$ is $0$.
The estimate is sharp in this case.
If $(r,s,t)=(2,0,1)$, then the circumradius-derived upper bound on the long edge is $\sqrt{32/3}$, the solid angle is at least
$\sol_0$, and the constant $d_3$ is $0$.  The estimate is sharp in this case as well.
If $(r,s,t)=(0,0,3)$, then an upper bound on the long edges is $3.27$ (the circumradius of a triangle
with sides $3.0,3.0,3.27$ is greater than $\sqrt3$), by Lexell the area of a spherical triangle with three (Euclidean) edges
in the interval $[3.0,3.27]$ is at least 
\[
\pi/2 > \sol_0 + d_3 = \sol_0 + 3 (0.27).
\]
By a simple \cc{HFBBNUL}{This was done in Mathematica.  Footnotes give randomly assigned tracking numbers 
that can be used to locate the details of the calculation at the website~\cite{website:FlyspeckProject}.}, 
the other cases have been checked in a simlar way.

In the case $(r,s,t)=(1,1,1)$, there is an edge length in the interval $\leftclosed 2 h_0,3.0\rightopen$.
If the upper bound on the longest edge is at most $3.45$, then the lower bound is calculated by Lexell
following the procedure just described.  Again, if the lower bound on the midlength edge
is at least $2.6$, then the procedure gives the bound.  However, if both of these conditions fail, we
have a triangle whose three edges fall in the intervals $[2,2]$, $[2h_0,2.6]$, $[3.45,2\sqrt3]$,
respectively.  Every
triangle with these edge bounds is obtuse and
has circumradius at least $\sqrt3$.  Thus, the situation is vacuous, and the bound \eqref{eqn:d3}
holds in the case $(r,s,t)=(1,1,1)$.


Next, we observe that the function $d_3$ is superadditive.  If we combine two adjacent regions with parameters
$(r_1,s_1,t_1)$ and $(r_2,s_2,t_2)$, the combined region has parameters $(r',s',t')$, where
\[
(r',s',t') =
\begin{cases} (r_1+r_2,s_1+s_2,t_1+t_2-2) &\text{ if the shared edge has length at least } 3.0\\
(r_1+r_2,s_1+s_2-2,t_1+t_2) &\text{ if the shared edge has length in } \leftclosed 2h_0,3.0\rightopen.
\end{cases}
\]
Then we can verify by a routine case-by-case calculation that
\[
d_3(r_1,s_1,t_1) + d_3(r_2,s_2,t_2) \ge d_3(r',s',t').
\]
(Note that $(r_i,s_i,t_i)\ne (3,0,0)$ and that we cannot have $(r_1,s_1,t_1)=(r_2,s_2,t_2)=(2,0,1)$ because
a quadrilateral of side length $2$ always has a diagonal of length at most $\sqrt8$.)

To complete the proof, we show that the bound \eqref{eqn:d3} gives the bound of the lemma.
Let $A_i$, for $i=1,\ldots (k-2)$, be the areas of the Delaunay triangles in the partition of $U_F$.
We write $(r_i,s_i,t_i)$ for the parameters of $A_i$ and return to the earlier notation of $(r,s)$ 
as the parameters defined
in the statement of the lemma.
By superadditivity, we have
\begin{align*}
\tau(U_F) &= (2 - k(F))\sol_0 + \sol(U_F)\\ 
    &= (2 - k(F))\sol_0 + \sum_i A_i  \\
   &= \sum_i (-\sol_0 + A_i)\\
   &\ge \sum_i d_3(r_i,s_i,t_i) \\
   &\ge d_3(r,s,0)\\
   &=d_2(r,s).
\end{align*}
The lemma ensues.\qed
\end{proof}

\section{Biconnected Fans}

We may create  fans that are biconnected graphs in the same way as in
\cite{Hales:2006:DCG}.  Here is a review
of the construction.

\begin{lemma}\guid{NJFWRPQ}\label{lemma:V'-bi} 
Let $V\in \CalV$.  Then there exists $V'\in \CalV'$ and a bijection 
 $\phi:V'\mapsto V$ that induces a bijection
  of contact graphs:
\[
\phi_*:(V,E_2(V)) \cong (V',E_2(V')).
\]
\end{lemma}

\begin{proof}
  Begin with the  fan $(V,E_2(V))$.  

  \claim{We claim that $(V,E^+(V))$ is a fan.} Indeed, it is checked by
  \cite[Lemma~4.30]{Hales:2006:DCG} that the blades satisfy the
  intersection property of fans, except possibly when two new blades
  are the diagonals of a quadrilateral face in $(V,E_2(V))$.  (The
  cited lemma uses the constant $2.51$ instead of $2h_0$, but this
  does not affect the reasoning of the lemma.) 
We may directly rule out the possibility of a quadrilateral face as follows.
 The diagonals of a quadrilateral face
  in $(V,E_2(V))$ is a spherical rhombus and one of its diagonals is
  necessarily at least $\sqrt8$ (with extreme case a square of side
  $2$).  The other fan properties are easily checked.

  If the hypermap $\op{hyp}(V,E^+(V))$ is not connected,
  the set of nodes $V_1\subset V$ in one combinatorial component can
  be moved closer to another combinatorial component until a new edge
  is formed.  This can be done in a way that the deformation of $V$
  remains in $\CalV$ and no new edges of length at most $2h_0$ are formed.
  Continuing in this fashion, a connected hypermap is obtained.

A biconnected hypermap is produced by  further
 deformations of the fan around each \newterm{articulation node}  (that is, a node 
whose deletion increases the number of combinatorial components).
\qed\end{proof}

\begin{definition}[$D_U$,~$m_U$,~$r_U$,~$s_U$,~$k_U$,~$\tau(U)$]
  Let $V\in \CalV$.  Let $U$ be a topological component of
  $Y(V,E_2(V))$ and let $D_U$ be the set of all darts of $\op{hyp}(V,E_2(V))$
 that lead into
  $U$.   For each
  $x\in D_U$, let $m(x) >0$ be the smallest positive integer such that
  $f^{m} x$ and $x$ lie at the same node.  Let $m_U$ be the maximum of
  $m(x)$  as $x$ runs over $D_U$.  The constant $m_U$ can be viewed as
  a \newterm{simplified face size}.  
Let $r_U$ be the number of nonisolated darts in $D_U$, and let $2+s_U$ be twice the
  number of combinatorial components of $\op{hyp}(V,E_2(V))$ that
  meet $D_U$.  (In particular, $s_U$ is even and $s_U=0$ exactly when $D_U$ lies in a 
single combinatorial component of the hypermap.)
Let $k_U=r_U+s_U$.  Overloading the symbol $\tau$, we
  set $\tau(U) = \sol(U) +  (2-k_U)\sol_0$.   (If a single  face $F$ leads into
  $U$ and if the face is simple, then the overloaded notation is consistent
with the earlier notation: $\tau(U) = \tau(V,E_2(V),F)$.)
  \indy{Index}{fan}%
  \indy{Notation}{m@$m_U$ (simplified face size)}%
  \indy{Notation}{s@$s_U$ (integer invariant of component)}%
  \indy{Notation}{r@$r_U$ (integer invariant of component)}%
  \indy{Notation}{k@$k_U$ (integer invariant of component)}%
  \indy{Notation}{D@$D_U$ (the set of darts leading into $U$)}%
\end{definition}

\begin{lemma}\guid{NOKWBKT}\label{lemma:tauU'}
  Let $V\in \CalV'$.
Let $U$ be a topological component of
  $Y(V,E_2(V))$.   Then $\tau(U)\ge \min(d_2(r_U,s_U),\op{tgt})$.
\end{lemma}

\begin{proof}
Up to a null set (given by the finite union of blades $C^0(\ee)$ for
$\ee\in E^+(V)\setminus E_2(V)$), the region $U$ is the union of topological
components $U_F$ of $Y(V,E^+(V))$, which are in bijection with the faces
$F$ of $\op{hyp}(V,E^+(V))$.  The function $\tau(U)$ is additive:
\begin{equation}\label{eqn:tau-additive}
\tau(U) = \sum_{U_F\subset U} \tau(V,E^+(V),F).
\end{equation}
By the biconnectedness of $(V,E^+(V))$, each value $\tau(V,E^+(V),F)$ is the
same before and after localization.  (Localization replaces $(V,E^+(V),F)$ with $(V',E',F)$
where $V'\subset V$, $E'\subset E^+(V)$, and $(V',E',F)$ is a local fan.)
Lemma~\ref{lemma:main-estimate-12} gives a lower bound on the
constants $\tau(V,E^+(V),F)$.  The constants $d_2(r_U,s_U)$ are superadditive:
\[
d_2(r_U,s_U) \le \sum_{U_F\subset U} d_2(r(V,F),s(V,F)),
\]
where $s(V,F)$ is the cardinality of the set of edges of $E^+(V)\setminus
E_2(V)$ that meet $F$, and $r(V,F) = \card(F)-s(V,F)$.  Thus, the 
lower bound on $\tau(U)$ follows from the main estimate
(Theorem~\ref{lemma:main-estimate-12}).
\qed\end{proof}

\begin{lemma}\guid{UCEUZYO}\label{lemma:uce} 
Let $V\in \CalV'$.  Then
\[
\sum_{U\in [Y(V,E_2(V))]} \tau(U) = (4\pi - 20\sol_0),
\]
where $[Y]$ is the set of topological components of $Y=Y(V,E_2(V))$.
\end{lemma}

\indy{Notation}{L1@$\CalL(V)$ (estimation of a packing)}%
\begin{proof} For a packing of twelve points $V\subset S^2(2)$, we have
$12=\CalL(V)$, where $\CalL(V)$ is the left-hand side of equation \eqref{eqn:L}.   
From this equality, following~\cite[8.2.3]{DSP}, we have
\[
  \sum_F \tau (V,E^+(V),F) = (4\pi - 20\sol_0),
\]
the sum running over faces of $\op{hyp}(V,E^+(V))$.
The result follows by additivity~\eqref{eqn:tau-additive}.
\qed\end{proof}

The constant $\op{tgt}=1.541$ is slightly
larger than $(4\pi-20\sol_0)\approx 1.54065$.  The constants $d_2(r_U,s_U)$ are nonnegative, so that
$\tau(U)$ is as well.  This means that for every subset
$A$ of $[Y(V,E_2(V))]$, we have
\begin{equation}\label{eqn:subtgt}
\sum_{U\in A} \tau(U) < \op{tgt}.
\end{equation}

\begin{lemma}[biconnected]\guid{BTZPFMU}\label{lemma:biconnected}
  Let $V\in \CalV$.  Then $\op{hyp}(V,E_2(V))$ is biconnected.
\end{lemma}

\begin{proof}
  By Lemma~\ref{lemma:V'-bi}, we may replace $V$ with a new set in
  $\CalV$ if necessary so that $(V,E^+(V))$ is a biconnected fan.
   We  show that the smaller
  fan $(V,E_2(V))$ is also biconnected.

  Let $U$ be a topological component of $Y(V,E_2(V))$.  Lemma~\ref{lemma:tauU'}
   implies that $\tau(U)\ge \min(d_2(r_U,s_U),\op{tgt})$.

   \claim{We claim that if $m_U\le 5$, then $D_U$ is a simple face.} Otherwise,
   either $D_U$ is a face that is not simple, or it consists of more than
   one face.  Either way, some node $\v$ lies in the interior to the
   $m_U$-gon.  Let $\u,\w$ be consecutive nodes around the $m_U$-gon.
   By a \cc{6621965370}
  {(Mathematica)} the angles
   $\op{azim}(\orz,\v,\u,\w)$ are each less than $2\pi/5$. The angles
   around $\v$ cannot sum to $2\pi$ as required.

   \claim{We claim that $D_U$ is a simple face.}  Otherwise, assume
   for a contradiction that $D_U$ is not simple, $m_U\ge 6$, and
   $d_2(r_U,s_U)<\op{tgt}$.  From the classification of
   \cite[p.~126,~Fig.~12.1]{Hales:2006:DCG}, and the inequalities
   $d_2(9,0) > \op{tgt}$, $d_2(6,2) > \op{tgt}$, it follows that 
    $m_U=6$ and $\tau(U)\ge d_2(8,0)$.  The set $D_U$ meets seven
   nodes: the six nodes counted by $m_U$ and a node in the interior of
   the hexagonal arrangement.  At each node there is a face of the
   hypermap $\op{hyp}(V,E_2(V))$ that is not a triangle, because
   $2\pi$ is not an integer multiple of the dihedral angle of a
   regular tetrahedron.  As the packing has twelve nodes in all, five
   nodes remain, each meeting a nontriangular topological component of
   $Y(V,E_2(V))$.  Thus, by counting nodes, the hypermap has at least
   one pentagon or two quadrilaterals.  We find that $\sum_{U}
   \tau(U)$ is at least
\[
d_2(8,0) + d_2(5,0) > \op{tgt}, \text{ or }\quad d_2(8,0) + 2 d_2(4,0) > \op{tgt},
\]
which is contrary to~\eqref{eqn:subtgt}.
Thus, $D_U$ is a simple face.
\indy{Index}{weight!total}%
\indy{Index}{weight}%

\claim{We claim that the hypermap is biconnected.}  Otherwise, if the hypermap is
not connected, then we can find two faces of the hypermap that lead
into the same topological component of $Y(V,E_2(V))$.  If the
hypermap is connected but not biconnected, then some face of the
hypermap is not simple.  Both possibilities contradict the fact that
$D_U$ is a simple face.
\qed\end{proof}

\section{Tame Contact}

This subsection defines a notion of tameness that includes hypermaps
that arise as the fan of a packing with kissing number twelve.  In the
definition of tame hypermap in \cite{DSP}, a function $b$
is used.  In this section we use a similar function,  which is
again called $b$.   Recall that $\op{tgt}=1.541$.  \indy{Index}{tame}%
\indy{Index}{hypermap!tame}%

\begin{definition}[b]\guid{IHRZTPV}
  Define $b:\ring{N}\times \ring{N}\to \ring{R}$ by
  $b\pqr{(p,q)}=\op{tgt}$, except for the following values:
\[
b(0,3)=b(1,3)=0.618,\quad b(2,2)=0.412.
\]
\end{definition}
\indy{Notation}{b@$b$ (contact weight parameter)}%

\begin{definition}[d]\guid{VUJQZCG}
Define $d_1:\ring{N}\to \ring{R}$ by
\[d_1(k) = \begin{cases}
0, & k\le 3, \\
0.206 + 0.27 (k-4),& k=4,\ldots,8,\\
\op{tgt},& k>8.\\
\end{cases}
\]
\end{definition}
\indy{Notation}{d@$d_1$ (real parameter)}%

The function $d_1$ is related to the two-variable function in
Lemma~\ref{lemma:main-estimate-12}: $d_1(k) = d_2(k,0)$, when $4\le k\le
8$.

We say that a node of a  fan $(V,E)$ has type $(p,q,r)\in\ring{N}^3$ if at the node there are
$p+q+r$ faces, of which $p$ are triangles and $q$ are quadrilaterals.

\indy{Index}{weight!contact assignment}%
  \indy{Index}{contact!weight assignment}%
\indy{Notation}{zzt@$\tau$ (weight assignment)}%
\begin{definition}[weight~assignment]\guid{GLIQSFM}
  A \fullterm{weight assignment\/}{weight!assignment} 
on a hypermap $H$ is a
  function $\tau$ on the set of faces of $H$ taking values in the set
  of nonnegative real numbers. A weight assignment $\tau$
is a \newterm{contact}
  weight assignment if the following two properties hold.
\begin{enumerate}
\item If the face $F$ has cardinality $k$, then
$\tau(F) \ge d_1(k)$.
\item If a node $\v$ has type $(p,q,0)$, then
  \[\sum_{F:\,\v\cap F\ne\emptyset} \tau(F) \ge
    b{\pqr{(p,q)}}.\]
\end{enumerate}
The sum $\sum_F \tau(F)$ is called the \fullterm{total weight}{weight!total} of $\tau$.
\end{definition}
\indy{Index}{total weight|see{weight}}%

\begin{definition}[tame contact]\guid{XJPQTIV}
  A hypermap has \fullterm{tame contact}{tame!contact} if it satisfies the following 
  conditions.
\indy{Index}{tame!contact}%
\indy{Index}{contact!tame}%
\indy{Index}{planar}%
\indy{Index}{biconnected}%
\indy{Index}{nondegenerate}%
\indy{Index}{loop}%
\indy{Index}{double join}%
\begin{enumerate}
\item \firstcase{biconnected} The hypermap is biconnected.  In particular,
  each face meets each node in at most one dart.
\item \firstcase{planar} The hypermap is plain and planar.  (A hypermap is defined
to be plain when $e$ is an involution on $D$; that is, $e^2 = I_D$. A connected hypermap
is planar when Euler's relation holds: $\#n + \#e + \#f = \card(D) + 2$, where 
for any permutation $h$, we write $\#h$ for
the number of orbits of $h$ on $D$.)
\item \firstcase{nondegenerate} The edge map $e$ has no fixed points.
\item \firstcase{no loops} The two darts of each edge lie in different
  nodes.
\item \firstcase{no double join} At most one edge meets any two given
  nodes.
\item \firstcase{face count} The hypermap has at least two faces.
\item \firstcase{face size} The cardinality of each face is at least three
  and at most eight.
\item \firstcase{node count} The hypermap has twelve nodes.
\item \firstcase{node size} The cardinality of every node is at least two  and at most four.
\item \firstcase{weights} There exists a contact weight assignment of total
  weight less than $\op{tgt}$.
\end{enumerate}
\end{definition}


\begin{theorem}\guid{ZXZSVPH} The contact hypermap $\op{hyp}(V,E_2(V))$ of a 
  packing $V\in \CalV$ is a hypermap with tame contact.
\end{theorem}
\indy{Index}{hypermap!tame}%
\indy{Index}{hypermap!contact}%
\indy{Index}{hypermap}%

\begin{proof} It is enough to go through the list of properties that
  define a tame contact hypermap and to verify that the contact
  hypermap satisfies each one.  We use the weight assignment $F\mapsto
  \tau(V,E_2(V),F)$.

\begin{enumerate}
\item \case{biconnected} The hypermap is biconnected by
  Lemma~\ref{lemma:biconnected}.
\item \case{planar} The contact hypermap is plain and planar by the
  general properties of fans.  (See \cite[Lemma~5.8]{DSP} and \cite[5.3]{DSP}.)
\item \case{nondegenerate} The
  edge map has no fixed points by the general properties of fans.
See \cite[Lemma~5.8]{DSP}.
\item \case{no loops} 
  There are no loops by the general properties of
  fans. See \cite[Lemma~5.8]{DSP}.
\item \case{no double join} This is also a general property of fans. See \cite[Lemma~5.8]{DSP}.
\item \case{face count} Each node has at least two darts by
  biconnectedness. Each face is simple; so the two darts at a node lie
  in different faces.  Thus, the hypermap has at least two faces.
\item \case{face size} The cardinality of each face is at least three
  because, as we have just observed, the hypermap has no loops or double joins.
  The cardinality of a face is at most eight because of the estimate
  $d_2(9,0)>\op{tgt}$.
\item \case{node count} There are twelve nodes by the definition of a
  packing with kissing number twelve.
\item \case{node size} We have already established that the cardinality
  of each node is at least two.  The proof that the cardinality is
  never  greater than four appears in Lemma~\ref{lemma:no-5}.
\item \case{weights} Theorem~\ref{lemma:main-estimate-12} establishes 
 the inequality $\tau(V,E_2(V),F)\ge d_1(k)$.
   \indy{Notation}{zzt@$\tau(V,E,F)$}%
  The total weight of the weight assignment is given by
  Lemma~\ref{lemma:uce}:
\[
  \sum_F \tau(V,E_2(V),F) = (4\pi - 20\sol_0) < \op{tgt}.
\]
\indy{Index}{weight!total}%
Let $\v$ be a node of type $(p,q,0)$.  
Then by the main estimate,
\[
\sum_{F\mid F\cap \v\ne\emptyset}\tau(V,E_2(V),F) > d_1(4)~q.
\]
This gives the nonzero entries in the table of bounds $b(p,q)$.  The
remaining entries follow from Lemma~\ref{lemma:no-5}.
\end{enumerate}
\qed\end{proof}

\begin{lemma}\guid{CQRHDZE}\label{lemma:no-5} 
  Let $V\in \CalV$.  Every node of $(V,E_2(V))$ has degree at most
  four.  Furthermore, suppose the type of a node is $(p,q,0)$.  Then
  $(p,q)$ must be
\[
(0,3),~(1,3),~\text{ or}~~(2,2).
\]
\end{lemma}

\begin{proof} The interior angles of a spherical polygon in the
  contact graph have the following lower $\alpha_k$ and upper bounds
  $\beta_k$, as a function of the number of sides $k$.
\begin{equation}
\begin{array}{lllll}
  \phantom{\ge}k~~&\alpha_k & \beta_k \vspace{6pt}\\
  \phantom{\ge}3~~&\op{dih}(2,2,2,2,2,2)  &\op{dih}(2,2,2,2,2,2)\\
  \phantom{\ge}4~~&\op{dih}(2,2,2,2h_0,2,2) &2\,\op{dih}(2,2,2,2,2h_0,2)\\
  {\ge}5~~& \op{dih}(2,2,2,2h_0,2,2) ~~~& 2\pi.
\end{array}
\end{equation}
Thus,
\[
  p\,\alpha_3 + q\,\alpha_4 +r\, \alpha_5 
\le 2\pi \le p\,\beta_3 + q\,\beta_4 + r \,\beta_5.
\]
There are no solutions for
$(p,q,r)$ in natural numbers when $p+q+r\ge 5$ and
 only the three given solutions in $(p,q,r)$ with $r=0$.
\qed\end{proof}

\section{Classification}

The website for the computer code  contains a list of eight hypermaps
that have been obtained by running the classification algorithm with
the tame contact parameters~\cite{website:FlyspeckProject}.

\begin{lemma}[tame hypermap classification]\guid{AZYOJBE}\cutrate{}
  \label{lemma:contact-classification} Every hypermap with tame
  contact is isomorphic to a hypermap in the given list of eight
  hypermaps, or is isomorphic to the opposite of a hypermap in the
  list.  \indy{Index}{isomorphism!tame contact classification}%
\end{lemma}

\begin{proof}
  By a \cc{PYWHMHQ}{}, 
 the set of all hypermaps has been classified by the same
  algorithm described in \cite[Chapter~4]{DSP}.
  \indy{Index}{contact!tame}%
  \indy{Index}{hypermap}%
  \indy{Index}{hypermap!tame}%
\qed\end{proof}

\begin{lemma}\guid{MWWSZTX}\label{lemma:fcc-ft} Let $V\in \CalV$.
  Suppose that $H=\op{hyp}(V,E_2(V))$ is a hypermap with tame
  contact.  Then $H$ is the FCC or HCP contact hypermap.
\end{lemma}

\begin{proof} The explicit enumeration of hypermaps with tame
  contact has eight cases.  Two are the hypermaps of the
  FCC and HCP.  The remaining six must be eliminated.  
A geometrical argument  eliminates one of these cases and linear programming
eliminates the other five.

\claim{We claim that one case with a hexagon cannot be realized geometrically as a contact fan
(Figure~\ref{fig:fthex}).}  Indeed, the perimeter of a hexagon with sides $\pi/3$
is $2\pi$.  However, the hexagons are geodesically convex,
 and $2\pi$ is a strict upper bound on the perimeter of the
hexagon.  Thus, this case does not exist.

\def\smalldot#1{\draw[fill=black] (#1) node [inner sep=1.3pt,shape=circle,fill=black] {}}
\def\graydot#1{\draw[fill=gray] (#1) node [inner sep=1.3pt,shape=circle,fill=gray] {}}
\def\whitedot#1{\draw[fill=gray] (#1) node [inner sep=1.3pt,shape=circle,fill=white,draw=black] {}}
\tikzset{dartstyle/.style={fill=black,rotate=-90,inner sep=0.7pt,dart,shape border uses incircle}}
\tikzset{grayfatpath/.style={line width=1ex,line cap=round,line join=round,draw=gray}}

\pgfmathsetmacro\cmofpt{(2.54/72.0)}

\def\figCXFENOK{
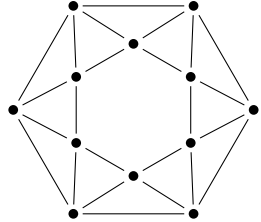
\tikzfig{fthex}{\guid{CXFENOK} This hypermap has tame contact but cannot be realized as a packing
in $\CalV$.}
{
[scale=0.004]
\path ( 400,0) node (P0) {};
\path (60:400)  node (P1) {};
\path (120:400) node (P2) {};
\path ( -400,0) node (P3) {};
\path ( -200,-346) node (P4) {};
\path ( 200,-346) node (P5) {};
\path (330:220) node (P6) {};
\path (270:220) node (P7) {};
\path(210:220) node (P8) {};
\path (30:220) node (P11) {};
\path (150:220) node (P9) {};
\path (90:220) node (P10) {}; 
\draw
  (P0) -- (P5)
  (P0) -- (P1)
  (P0) -- (P11)
  (P0) -- (P6)
  (P1) -- (P2)
  (P1) -- (P10)
  (P1) -- (P11)
  (P2) -- (P3)
  (P2) -- (P9)
  (P2) -- (P10)
  (P3) -- (P4)
  (P3) -- (P8)
  (P3) -- (P9)
  (P4) -- (P5)
  (P4) -- (P7)
  (P4) -- (P8)
  (P5) -- (P6)
  (P5) -- (P7)
  (P6) -- (P11)
  (P6) -- (P7)
  (P7) -- (P8)
  (P8) -- (P9)
  (P9) -- (P10)
  (P10) -- (P11)
;
\foreach \i in {P0,P1,P2,P3,P4,P5,P6,P7,P8,P9,P10,P11} {
  \smalldot {\i};
}
}}

\figCXFENOK 

There are some linear
  programming constraints that are immediately available to us.
\begin{enumerate}\wasitemize 
\item The angles around each node sum to $2\pi$.
\item Each angle of a triangle is $\alpha_3$.
\item Each angle of each rhombus lies between $\alpha_4$ and $\beta_4$.
\item The opposite angles of each rhombus are equal.
\end{enumerate}\wasitemize 
By a linear programming \cc{JKJNYAA}{},
these systems of constraints are infeasible in the remaining five cases.
\qed\end{proof}

\begin{lemma}\guid{YRTPQXK}\label{lemma:kiss-fcc}
  Let $V\in \CalV$ be a packing such that $\op{hyp}(V,E_2(V))$ is
  isomorphic to the FCC or HCP contact hypermap.  Then $V$ is
  congruent to the FCC or HCP configuration in $S^2(2)$.
\end{lemma}
\indy{Index}{HCP}%
\indy{Index}{FCC}%
\indy{Index}{kissing configuration}%
\indy{Index}{hypermap}%

\begin{proof} Every face of the hypermap of $(V,E_2(V))$ is a
  triangle or quadrilateral.  The eight triangles in the FCC or HCP
  contact hypermap determine eight equilateral triangles in $V$ of
  edge length $2$.  The eight triangles rigidly determine $V$ up to
  congruence.
\qed\end{proof}

\begin{proof}[Proof of Theorem~\ref{thm:fc}]  
  The contact hypermap of a packing with kissing number twelve has tame
  contact.  By Theorem~\ref{lemma:fcc-ft}, this hypermap is that of
  the FCC or HCP.  By Lemma~\ref{lemma:kiss-fcc}, the kissing
  configuration of the packing is congruent to the FCC or HCP.  As the
  center of the packing may be chosen at an arbitrary point in the
  packing, every point in the packing is congruent to one of these two
  arrangements.  The result ensues.
\qed\end{proof}

\raggedright
\bibliographystyle{amsalpha} 
\bibliography{all}

\providecommand{\bysame}{\leavevmode\hbox to3em{\hrulefill}\thinspace}
\providecommand{\MR}{\relax\ifhmode\unskip\space\fi MR }
\providecommand{\MRhref}[2]{%
  \href{http://www.ams.org/mathscinet-getitem?mr=#1}{#2}
}
\providecommand{\href}[2]{#2}
\begin{thebibliography}{{Fej}89}

\bibitem[{Fej}69]{Fejes-Toth:69}
L.~{Fejes T{\'o}th}, \emph{Remarks on a theorem of {R. M.} {R}obinson}, Studia
  Scientiarum Hungarica \textbf{4} (1969), 441--445.

\bibitem[{Fej}89]{Fejes-Toth:89}
\bysame, \emph{Research problems}, Periodica Mathematica Hungarica \textbf{29}
  (1989), 89--91.

\bibitem[Hal12a]{DSP}
T.~C. Hales, \emph{Dense sphere packings: a blueprint for formal proofs},
  London Math Soc. Lecture Note Series, vol. 400, Cambridge University Press,
  2012.

\bibitem[Hal12b]{website:FlyspeckProject}
\bysame, \emph{{The Flyspeck Project}}, 2012,
  \url{http://code.google.com/p/flyspeck}.

\bibitem[HF06]{Hales:2006:DCG}
T.~C. Hales and S.~P. Ferguson, \emph{The {Kepler} conjecture}, Discrete and
  Computational Geometry \textbf{36} (2006), no.~1, 1--269.

\bibitem[Lee56]{Leech:1956:MG}
J.~Leech, \emph{The problem of the thirteen spheres}, Mathematical Gazette,
  February 1956, pp.~22--23.

\end{thebibliography}



\end{document}